\documentclass{amsart}

\usepackage{mathpazo}
\usepackage[scaled=.95]{helvet}
\usepackage{courier}
\advance \topmargin by -\headheight
\advance \topmargin by -\headsep
\evensidemargin \oddsidemargin
\theoremstyle{plain}
\newtheorem{thm}{Theorem}[section]
\newtheorem{prop}[thm]{Proposition}
\newtheorem{lemma}[thm]{Lemma}
\newtheorem{cor}[thm]{Corollary}
\theoremstyle{definition}
\newtheorem{defn}{Definition}
\theoremstyle{remark}
\newtheorem{remark}[thm]{Remark}
\newcommand{\bR}{{\mathbb{R}}}
\newcommand{\bZ}{{\mathbb{Z}}}
\newcommand{\bT}{{\mathbb{T}}}
\newcommand{\ttimes}{{\widetilde{\times}}}

\title{Involutions on tori with codimension-one fixed point set}
\author{Allan L. Edmonds}
\address{Department of Mathematics, Indiana University, Bloomington, IN 47405}
\email{edmonds@indiana.edu}
\date{}
\begin{document}
\begin{abstract}
The standard P. A. Smith theory of $p$-group actions on spheres, disks, and euclidean spaces is extended to the case of $p$-group actions on tori (i.e., products of circles) and coupled with topological surgery theory to give a complete topological classification, valid in all dimensions, of the locally linear, orientation-reversing, involutions on tori with fixed point set of codimension one.
\end{abstract}
\maketitle
\section{Introduction}We extend the standard P.~A.~Smith theory of $p$-group actions on spheres, disks, and euclidean spaces to the case of $p$-group actions on tori  $T^{n}=S^{1}\times S^{1}\times \dots \times S^{1}$ ($n$ factors). Then we  apply the topological surgery machine to give a complete topological classification  of locally linear actions of the group $C_{2}$ order $2$ with non-connected, codimension-one
 fixed point set.

The simplest standard model of such an action of $C_{2}$ is the action obtained as the cartesian product of the trivial action on $T^{n-1}$ with the action on the circle $S^{1}$ fixing two points. Its fixed point set consists of two copies of $T^{n-1}$, which together separate $T^{n}$ into two copies of $T^{n-1}\times I$. The action of the generating involution may then be described as the map of the double  $D(T^{n-1}\times I)=T^{n-1}\times I \cup T^{n-1}\times I$ (identified along their common boundaries by the identity) that interchanges the two summands. Another model action is that where the generator interchanges two coordinates of $T^{n}$, fixing a single copy of $T^{n-1}$. In this case the orbit space of the action can be described as the non-orientable, or twisted, $I$-bundle over $T^{n-1}$. Such an $I$-bundle is determined by an epimorphism $\pi_{1}(T^{n-1})\to \bZ_{2}$, and any two such twisted $I$-bundles are equivalent, allowing homemorphisms of the base torus.  By analogy with terminology in the topology of surfaces, we call such a twisted $I$-bundle a \emph{M\"obius band}.

We will show that for a general locally linear involution on $T^{n}$ with fixed point set of codimension-one, the fixed point set must consist either one or two $\bZ[\bZ^{n-1}]$-homology $(n-1)$-tori. We will  refer to the number of components of the fixed point set as the \emph{Type} of the action.

Then we give a complete analysis of the case of two components, showing that any two $\bZ[\bZ^{n-1}]$-homology $(n-1)$-tori arise as the fixed point set of a Type 2 action, and that two such actions with the same fixed point set must be equivalent. Similarly any single $\bZ[\bZ^{n-1}]$-homology $(n-1)$-torus is the fixed point set of a Type 1 actionion $T^{n}$, and any two Type 1 actions with the same fixed point set are equivalent.

\subsection*{Acknowledgement} Thanks to  Michal Sadowski for pointing  out an error in an earlier version of this work.

\section{Smith theory for $p$-group actions on tori}
The most basic P.~A.  Smith theory, as described, for example, by G. Bredon, \cite{Bredon1972}, Chapter 3, implies that the fixed point set of a $p$-group acting on euclidean space is $\bZ_{p}$-acyclic. We apply the technique of lifting a group action to the universal covering space, perhaps first used by P. Conner and D. Montgomery \cite{ConnerMontgomery1959} and heavily exploited by Conner and F. Raymond \cite{ConnerRaymond1970}.
\begin{thm}[Homology torus fixed set of constant dimension]
If a finite $p$-group $G$ acts on the $n$-torus, then each component of the fixed point set has the mod $p$ homology of a $k$-torus for some $k$, and, in fact, the $\bZ_{p}[\bZ^{k}]$-homology of $T^{k}$. Moreover, all components of the fixed point set have the same dimension.
\end{thm}
\begin{proof}
Let $x$ be a point of the fixed point set $F$.  We may lift the action of $G$ to a covering action on $\mathbb{R}^{n}$ uniquely determined by the requirement that it fix a chosen point lying over the point $x\in F$. By Smith theory, the fixed point set $\widetilde{F}$ of $G$ acting on $\mathbb{R}^{n}$ is a $\bZ_{p}$-acyclic $\bZ_{p}$-homology $k$-manifold for some $k\le n$. Moreover, $\widetilde{F}$ projects as a covering map into $F$ with its image coinciding with the component $F_{x}$ of $F$ in which $x$ lies. The group of deck transformations of the regular covering $\widetilde{F}\to F_{x}$ consists of the subgroup of the group $\bZ^{n}$ of deck transformations for $\bR^{n}\to T^{n}$ that leave $\widetilde{F}$ invariant. In particular it is a free abelian group of rank $k$ for some $k\le n$. The spectral sequence of the covering $\widetilde{F}\to F_{x}$ (with $\bZ_{p}$-coefficients) shows that $H_{*}(F_{x};\bZ_{p})\approx H_{*}(T^{k};\bZ_{p})$, and that, indeed, almost by definition, $H_{*}(F_{x};\bZ_{p}[\bZ^{k}])\approx H_{*}(T^{k};\bZ_{p}[\bZ^{k})$.

It remains to see that all components of the fixed point set have the same dimension.
To this end, consider again the covering $\widetilde{F}\to F_{x}$ arising by choosing a fixed point $x$ and lifting the group action to $\bR^{n}$, fixing a chosen point over $x$. One can identify the group of deck transformations with the invariant elements $\pi_{1}(T^{n},x)^{G}$. The dimension of $F_{x}$ then is the mod $p$ cohomological dimension of this group. (See Brown \cite{Brown1994}, for example, for information about cohomological dimension.) But, since $\pi_{1}(T^{n},x)$ is abelian, the latter group, as well as the action of $G$ on it, is independent of the choice of fixed base point. The result follows.
\end{proof}
\begin{cor}[Nontrivial homology]
If a finite $p$-group acts on an $n$-torus, then each $k$-dimensional component $F_{x}$ of the fixed point set carries the non-zero mod $p$ homology class of a standard $k$-sub-torus of $T^{n}$.
\end{cor}
\begin{proof}
As above lift the given action to one on the universal covering $\bR^{n}$. The image of $\pi_{1}(F_{x},x)$ in $\pi_{1}(T^{n},x)$ must be $\pi_{1}(T^{n},x)^{G}$, which (being a fixed point set) is a direct summand of $\pi_{1}(T^{n},x)=\bZ^{n}$. Therefore each component $F_{x}$ carries the nontrivial mod $p$ homology class of a standard $k$-sub-torus $T^{k}$, since the covering $\widetilde{F}\to F_{x}$ is classified by a map $F_{x}\to T^{k}$ factorizing the inclusion $F_{x}\to T^{n}$ up to homotopy and inducing an isomorphism $H_{*}(F_{x};\bZ_{p})\to H_{*}(T^{k};\bZ_{p})$.
\end{proof}
Here is an alternative approach to the results of this section. An action of $G$ on $T^{n}$ determines a geometric model action of $G$ on $T^{n}$, which we denote briefly by $\bT^{n}_{G}$, by Lee and Raymond \cite{LeeRaymond1981}. There is then a $G$-map $T^{n}\to \bT^{n}_{G}$ inducing an isomorphism on $\pi_{1}$, as follows from a construction that perhaps goes back to Serre. The best way to see this is by lifting both actions to the universal covers and producing an equivariant map at that level by trivial obstruction theory, using the fact that the model action has contractible fixed point sets.  Then we can apply ordinary relative Smith Theory to the pair $(\bT^{n}_{G},T^{n})$, i.e., to the mapping cylinder relative to the domain, to obtain the desired conclusions.
\section{Involutions}%
Now consider the situation of orientation-reversing actions of the group $C_{2}$ of order two on the $n$-torus $T^{n}$ such that the fixed point set has dimension $n-1$.

As proved above each component $F_{x}$ of the fixed point set has the $\bZ_{2}[\bZ^{n-1}]$-homology of $T^{n-1}$. We will argue that $F_{x}$ is orientable, that there are exactly two components of the fixed point set, and that in fact each component has the $\bZ[\bZ^{n-1}]$-homology of $T^{n-1}$.

\begin{lemma}[Orientability]
If $C_{2}$ acts on the $n$-torus with codimension-one fixed point set, then each component of the fixed point set is orientable.
\end{lemma}
\begin{proof}

Consider the covering $\widetilde{F}\to F_{x}$, with its deck transformation group a summand $\bZ^{n-1}\subset \bZ^{n}$. Note that $\widetilde{F}$, being mod ${2}$ acyclic, is certainly orientable. If $F_{x}$ were non-orientable, then the action of $\bZ^{n-1}$ on $\widetilde{F}$ must reverse orientation. But of course $\bZ^{n-1}$ preserves orientation on all of $\bR^{n}$. thus the action of $\bZ^{n-1}$ interchanges sides of $\widetilde{F}$ in $\bR^{n}$. It follows that $\bR^{n}/\bZ^{n-1}$ is an orientable, non-compact manifold with boundary $F_{x}$. This implies that the boundary, namely $F_{x}$, is also orientable, contradicting the assumption that $F_{x}$ is non-orientable.
\end{proof}
\begin{lemma}[One or two components]
If $C_{2}$ acts on the $n$-torus with codimension-one fixed point set, then the fixed point set contains either one or two components.
\end{lemma}
\begin{proof}
By the basic inequality of 
Smith theory
\[
\sum_{i\ge 0}\dim_{\bZ_{2}}H_{i}(F;\bZ_{2})\le \sum_{i\ge 0}\dim_{\bZ_{2}}H_{i}(T^{n};\bZ_{2})
\]
Thus
\[
2^{n-1}b_{0}(F)\le 2^{n}
\]
It follows that there are at most two components. 
\end{proof}
We note that this also follows from more general formulas for the number of components of a fixed point set on tori (see M. Sadowski \cite{Sadowski2006}) or other aspherical manifolds (see Conner and Raymond \cite{ConnerRaymond1972}).

In this case there are regular coverings of each fixed point component $F_{i}$  with deck transformation group isomorphic to $\bZ^{n-1}$, that are $\bZ_{2}$-acyclic. If there are two components, they separate $T^{n}$ into two complementary domains interchanged by the group action. The closure of either complementary domain is homeomorphic to the orbit  space.  If there is only one component of the fixed point set, then it is nonseparating. The orbit space is a nonorientable manifold with boundary $F$, whose interior is covered $2$ to $1$ by the complement of $F$ in $T^{n}$.

The codimension-one
 aspect allows us to do a bit better, gleaning integral, not just mod $2$, information.

\begin{prop}[{$\bZ[\bZ^{n-1}]$-homology}]
If  the group $C_{2}$ of order two acts on the $n$-torus $T^{n}$ such that the fixed point set has dimension $n-1$, then any component $F_{i}$ of the fixed point set has the $\bZ[\bZ^{n-1}]$-homology of $T^{n-1}$. In particular there is a regular covering of $F_{i}$  with group $\bZ^{n-1}$ that is $\bZ$-acyclic. Moreover, the orbit space $W^{n}$ 
 also has the $\bZ[\bZ^{n-1}]$-homology of $T^{n-1}$ and has $\pi_{1}(W)\approx \bZ^{n-1}$.
\end{prop}
\begin{proof}
As we have seen, one may lift the action of $C_{2}$ to a covering action on $\mathbb{R}^{n}$ whose fixed point set $\widetilde{F}$ covers (one component of) $F\subset T^{n}$. And $\widetilde{F}$  is a $\bZ_{2}$-acyclic $\bZ_{2}$-homology $(n-1)$-manifold, by basic Smith theory, and the group of deck transformations preserving $\widetilde{F}$ is isomorphic to $\bZ^{n-1}$ and a summand of $\pi_{1}(T^{n})$. By duality such a mod $2$ hyperplane $\widetilde{F}$ separates $\bR^{n}$ into two components $U$ and $V$.  The involution in $C_{2}$ allows one to define retractions of $\bR^{n}$ onto the closures $\overline{U}$ and $\overline{V}$ of the complementary domains.  It follows that $\overline{U}$ and $\overline{V}$ are acyclic over $\bZ$ and have trivial fundamental group. Then a Mayer-Vietoris sequence argument implies that $\widetilde{F}$ is also acyclic. (Technical note: either one needs to assume the action is ``nice'' or that one is using, say, Cech cohomology.) It thus follows that (each component of) $F$ itself has the  $\bZ[\bZ^{n-1}]$-homology of $T^{n-1}$.

It remains to discuss the homology of $W$.  The full action of $\bZ^{n}$ on $\bR^{n}$ creates a $\bZ^{n}/\bZ^{n-1}=\bZ$ orbit of pairwise disjoint ``parallel'' copies of $\widetilde{F}$, separating $\bR^{n}$ into a sequence of ``strip'' domains $U_{i}$. The closures $\overline{U}_{i}$ of these strip domains are all acyclic, simply connected by van Kampen's theorem, invariant precisely under $\pi_{1}(T^{n})^{C_{2}}$, and cover a  complementary domain in $T^{n}$. It follows that both complementary domain(s) there have the $\bZ[\bZ^{n-1}]$-homology of $T^{n-1}$, and have $\pi_{1}=\bZ^{n-1}$. In the case where the fixed point set has two components, and two complementary domains interchanged by the involution, this describes the orbit space as well.

Finally we must complete the argument in the case when the fixed point set is connected and has a single complementary domain. Then the (interior of) the orbit space is covered two-to-one by the complement of the fixed set in $T^{n}$. In this case $\text{int}W$ is necessarily nonorientable, with orientable double covering given by $T^{n}-F$. It is necessary to note that the action of $C_{2}$ on $H_{1}(T^{n}-F)=\bZ^{n-1}$ is trivial.  Indeed, $H_{1}(T^{n}-F)$ and $H_{1}(F)$ coincide in $H_{1}(T^{n})$. Also $W$ is aspherical since it is covered by a contractible strip domain in $\bR^{n}$. It follows that $\pi_{1}(W)$ is a torsion-free central extension of $\bZ^{n-1}$ by $C_{2}$, hence isomorphic to $\bZ^{n-1}$. 
\end{proof}

\begin{remark}
Note that in the Type 1 case, the orientable double covering $T^{n}-F\to W$ is trivial over the image of $\pi_{1}(\partial W)$, and completely determined by this condition.
\end{remark}
\begin{remark}
When $n=3$ (and the action is locally linear) we observe that the orbit space $W^{3}$ is an irreducible $3$-manifold. Any embedded $2$-sphere would be trivially covered by a pair of $2$-spheres in $T^{3}$. Since $T^{3}$ is irreducible, these $2$-spheres must bound $3$-balls in $T^{3}$. It follows that the original $2$-sphere in $W^{3}$ also must bound  a ball.
\end{remark}

\begin{remark}
The observation that one obtains integral, not just mod $2$, information about codimension-one
 fixed sets and their complementary domains was perhaps first observed by the author and the late D. Galewski in \cite{EdmondsGalewski1976}, in the context of PL, not necessarily locally linear, actions on spheres.
 \end{remark}

Our goal now becomes one of showing that any homology $(n-1)$-torus or pair of homology $(n-1)$-tori arise as fixed point sets of locally linear involutions, and that any two such involutions with the same fixed point set are equivalent.
\section{Classification of homology tori}
Here we describe the classification of the sort of homology tori that appear as codimension-one fixed point sets in standard tori.
\begin{defn}
A $\bZ[\bZ^{n}]$-homology $n$-torus is a closed orientable $n$-manifold $M^{n}$ with the properties that $H_{1}(M_{n};\bZ)=\bZ^{n}$ and $\widetilde{H}_{*}(\widetilde{M}^{n};\bZ)=0$, where $\widetilde{M}^{n}$ denotes the universal abelian cover of $M^{n}$ (with deck transformation group $\bZ^{n}$) and $\widetilde{H}$ denotes reduced homology.
\end{defn}
In dimensions at least $3$ one can obtain simple nontrivial  examples in the form  $T^{n}\#\Sigma^{n}$, where $\Sigma^{n}$ is a non-simply connected integral homology  sphere. With more work one can construct interesting examples that do not split in such a simple  way.

Since the torus $T^{n}$ is aspherical, for any $\bZ[\bZ^{n}]$-homology $n$-torus $M^{n}$ there is a map $f:M^{n}\to T^{n}$ inducing an isomorphism of $H_{1}$, indeed all $H_{k}$, and of homology with local coefficients $\bZ[\bZ^{n}]$, well-defined up to homotopy and composition with a self-homotopy  equivalence of $T^{n}$.

\begin{defn}
A $\bZ[\bZ^{n}]$-homology-cobordism (of homology $n$-tori) is an $(n+1)$-mani\-fold $W^{n+1}$ with two boundary  components, $X^{n}$ and $Y^{n}$ such that all three spaces have compatible maps to $T^{n}$ inducing isomorphisms of homology with  local coefficients $\bZ[\bZ^{n}]$.  In particular $$H_{*}(W^{n+1},X^{n};\bZ[\bZ^{n}])=0=H_{*}(W^{n+1},Y^{n};\bZ[\bZ^{n}])$$
\end{defn}

If such a cobordism exists we say that $X^{n}$ and $Y^{n}$ are $\bZ[\bZ^{n}]$-homology cobordant. A  $\bZ[\bZ^{n}]$-homology-cobordism $W^{n+1}$ will be called a \emph{strong} $\bZ[\bZ^{n}]$-homology cobordism if $\pi_{1}(W^{n+1})\approx \bZ^{n}$.

We will apply the following two results that generalize the standard topological surgery classification of homotopy tori to the context of homology tori.
\begin{prop}\label{prop:cobordant}
Any two  $\bZ[\bZ^{n}]$-homology $n$-tori are strongly $\bZ[\bZ^{n}]$-homology cobordant.
\end{prop}
\begin{proof}
It suffices to show that any  $\bZ[\bZ^{n}]$-homology $n$-torus $X^{n}$ is strongly $\bZ[\bZ^{n}]$-homology cobordant to the standard torus $T^{n}$.

For $n\le 2$ this is true by the classification of $1$- and $2$-manifolds. 

For $n\ge 4$ it is an immediate consequence of the ``Plus Construction'' of Freedman and Quinn \cite{FreedmanQuinn1990}, 11.1A (dimension $4$) and 11.2 (higher dimensions). This requires noting that the kernel of the abelianization map $\pi_{1}\to \bZ^{n}$ is perfect and moreover that $\bZ^{n}$ is ``good'' (required only in dimension $4$).  The Plus construction describes a homology  cobordism with $\pi_{1}=\bZ^{n}$ to a homology torus with $\pi_{1}=\bZ^{n}$, and the latter is homeomorphic to $T^{n}$, by the classification of homotopy tori.

For $n=3$ this is  a special case of Theorem 15 of Jahren and Kwasik \cite{JahrenKwasik2003}, who prove that the $\bZ[\pi_{1}(M^{3})]$-homology structure set of a closed aspherical $3$-manifold is trivial in the cases when the manifold is Seifert fibered, hyperbolic or Haken with at least one hyperbolic piece in its torus decomposition. In our case we have $M^{3}=T^{3}$, which is certainly Seifert fibered. This line of reasoning requires a version of the surgery exact sequence for homology  equivalences, and the use of periodicity to move into higher dimensions, finally quoting higher dimensional rigidity results of Farrell and Jones, Leeb, and Stark. This theory produces a $\bZ[\bZ^{n}]$-homology cobordism. The Plus construction, applied to the cobordism gives a strong $\bZ[\bZ^{n}]$-homology cobordism.
\end{proof}
\begin{remark}
We outline a somewhat less-learned approach for the special case of $3$-dimensional homology tori. First we need to note a priori that any $\bZ[\bZ^{n}]$-homology equivalence $X^{3}\to T^{3}$  is normally cobordant to the identity $\text{id}:T^{3}\to T^{3}$. This is in fact Theorem 2 in \cite{JahrenKwasik2003}, which uses simply the existence of the surgery machine and an explicit calculation of low dimensional normal invariants to prove that the surgery obstruction map is a split monomorphism. It remains to justify that the surgery obstruction of the normal cobordism can be made to vanish. Let $F:W^{4}\to T^{3}\times I$ be a normal map. The Wall group $L_{4}(\bZ^{3})\approx \bZ\oplus\bZ_{2}^{3}$ by the Wall-Shaneson product  formula. The $\bZ$ is given by signature $/8$ and the $\bZ_{2}$ terms are given by codimension $2$ Arf invariants. We can kill the signature by connected sum with a suitable number of copies of the $\pm E_{8}$ manifold. Similarly we may change any nonzero Arf invariants by replacing a tubular neighborhood of a transverse preimage of a standard $2$-torus, of the form $F^{2}\times \text{int}D^{2}$, with $F^{2}\times (T^{2}-\text{int}D^{2})$, where the $T^{2}$ factor is given the framing with non-zero Arf invariant. Compare the argument of J. Davis \cite{Davis2006}, proof of the theorem. Then topological surgery can be carried out on the modified $4$-manifold, since the surgery obstruction vanishes and the fundamental group of the target is ``good'',  to produce the required strong $\bZ[\bZ^{3}]$-homology cobordism.
\end{remark}
\begin{prop}\label{prop:type2homeomorphic}
Any two strong $\bZ[\bZ^{n}]$-homology cobordisms (irreducible if $n+1=3$) between the same  pair of $\bZ[\bZ^{n}]$-homology $n$-tori are homeomorphic.
\end{prop}
\begin{proof}
For $n+1\ge 4$ this is an immediate consequence of the uniqueness clause in the Freedman-Quinn Plus Construction, \cite{FreedmanQuinn1990}, p. 197. For $n+1=3$ it is a special case of the $h$-cobordism theorem for Haken $3$-manifolds. And for $n+1=2$ it is a trivial consequence of the classification of surfaces.
\end{proof}
\begin{remark}
We could drop the irreducibility hypothesis when $n+1=3$ by invoking the Poincar\'{e} Conjecture as proved by G. Perelman. But since irreducibility is an easily verified necessary condition it seems reasonable simply to assume it. \end{remark}
For the classification of Type 1 involutions we need similar results where strong $\bZ[\bZ^{n}]$-homology cobordisms are replaced by what we shall call strong $\bZ[\bZ^{n}]$-homology M\"obius bands.

\begin{defn}
A $\bZ[\bZ^{n}]$-homology M\"obius band is a nonorientable $(n+1)$-manifold $W^{n+1}$ with one boundary  component, $X^{n}$, such that both $X^{n}$ and $W^{n+1}$ have the $\bZ[\bZ^{n}]$ homology of $T^{n}$, and the the inclusion induced homomorphism $H_{1}(X^{n})\to H_{1}(W^{n+1})$ is injective with image of index $2$. The homology M\"obius band is called strong if in addition $\pi_{1}(W^{n+1})\approx \bZ^{n}$.
\end{defn}

In this context we have the analogues of the existence and topological uniqueness of strong homology  cobordisms of homology tori, as follows.
\begin{prop}\label{prop:mobius}
Any   $\bZ[\bZ^{n}]$-homology $n$-torus $X^{n}$ is the boundary of a strong $\bZ[\bZ^{n}]$-homology M\"obius band.
\end{prop}
\begin{proof}
Just attach a strong $\bZ[\bZ^{n}]$-homology cobordism between $X^{n}$ and $T^{n}$ to the standard M\"obius band $T^{n}\ttimes I$ along the boundary $T^{n}$.
\end{proof}
\begin{lemma}\label{lemma:mobius}
Let $W^{n+1}$ be a strong M\"obius band with boundary $T^{n}$. (Assume $W^{n+1}$ is irreducible if $n+1=3$.) Then $W^{n+1}$ is homeomorphic to $T^{n}\ttimes I$.
\end{lemma}
\begin{proof}
What we need, from a topological surgery point of view, is for the topological structure set $\mathcal{S}(T^{n}\ttimes I)$ (rel boundary) to vanish. This  follows in high dimensions from the calculation of the surgery obstruction groups of $\pi_{1}=\bZ^{n}$ and the fact that topological surgery ``works'' when $n+1\ge 5$. For detailed treatment, see Kirby and Siebenmann \cite{KirbySiebenmann1977}, Appendix C, especially Theorems C.2 and C.7, where Theorem C.7 in particular allows nontrivial disk bundles over tori.

The same surgery argument applies when $n+1=4$, by Freedman and Quinn \cite{FreedmanQuinn1990}, since the fundamental groups in question are good.

When $n+1=3$, this  follows from standard Waldhausen theory of sufficiently large $3$-manifolds, since a M\"obius band is Haken. In dimension $n+1=2$, it is a consequence of the classification of surfaces.
\end{proof}
\begin{prop}\label{prop:type1homeomorphic}
Any two strong $\bZ[\bZ^{n}]$-homology M\"obius bands (irreducible if $n+1=3$) with the same   $\bZ[\bZ^{n}]$-homology $n$-torus as boundary are homeomorphic.
\end{prop}
\begin{proof}
Let $W_{1}^{n+1}$ and $W_{2}^{n+1}$ be two strong $\bZ[\bZ^{n}]$-homology M\"obius bands (irreducible if $n+1=3$) with the same   $\bZ[\bZ^{n}]$-homology $n$-torus $X^{n}$ as boundary.  Also let $V^{n+1}$ be the unique strong $\bZ[\bZ^{n}]$-homology cobordism between $X^{n}$ and $T^{n}$. Consider $W_{i}^{n+1}\cup_{X^{n}}V^{n+1}$. By Lemma \ref{lemma:mobius}, $W_{i}^{n+1}\cup_{X^{n}}V^{n+1}\cong T^{n}\ttimes I$. Then we may view $(T^{n}\ttimes I)\times I$ as a strong $\bZ[\bZ^{n}]$-homology cobordism between $W_{1}^{n+1}$ and $W_{2}^{n+1}$. But over the boundary we have $V^{n}\cup_{T^{n}}V^{n}$ between $X^{n}$ and $X^{n}$. Applying the plus construction to $V^{n}\cup_{T^{n}}V^{n}$, we augment $(T^{n}\ttimes I)\times I$ to an actual $s$-cobordism between $W_{1}^{n+1}$ and $W_{2}^{n+1}$. Thus the result follows from the $s$-cobordism theorem. This requires $n+2\ge 5$ or $n+1\ge 4$.

It remains to consider the low-dimensional cases where $n+1\le 3$. In these cases the boundary is a standard torus, and the result follows from Lemma \ref{lemma:mobius}.
\end{proof}
\begin{remark}
Note that the orientable double cover of a strong M\"obius band with boundary $X^{n}$ is the unique  strong $\bZ[\bZ^{n}]$-homology cobordism of $X^{n}$ to itself. It follows that a strong $\bZ[\bZ^{n}]$-homology cobordism from $X^{n}$ to itself admits a unique free, orientation-reversing, homeomorphism exchanging boundary components.
\end{remark}
\section{Classification of involutions}
Finally we interpret the preceding classification of homology tori in the context of involutions with codimension-one fixed point set.
\begin{prop}
If $X^{n-1}$ and $Y^{n-1}$ are $\bZ[\bZ^{n-1}]$-homology $(n-1)$-tori and are $\bZ [\bZ^{n-1}]$-homology cobordant, by a strong (irreducible) $\bZ[\bZ^{n-1}]$-homology cobordism $W^{n}$, then the group $G=C_{2}$  acts on the $n$-torus $T^{n}$ with fixed point set  homeomorphic to $X^{n-1}\cup Y^{n-1}$, and with orbit space $W^{n}$.
\end{prop}
\begin{proof}
The double of $W^{n}$ clearly admits an involution with fixed point set $X^{n-1}\cup Y^{n-1}$, and with orbit space $W^{n}$. The double is easily seen to be a homotopy torus, hence be homeomorphic to the standard torus.\end{proof}

Since the  strong (irreducible) $\bZ[\bZ^{n-1}]$-homology cobordism between two $\bZ[\bZ^{n-1}]$-homology $(n-1)$-tori is unique, by Proposition \ref{prop:type2homeomorphic} we have the following.

\begin{thm}The set of equivariant homeomorphism classes of locally linear involutions on $T^{n}$ with  non-connected, codimension-one
 fixed point sets is in one-to-one correspondence with the set of unordered pairs  $\{X^{n-1},Y^{n-1}\}$ of homeomorphism classes of $\bZ[\bZ^{n-1}]$-homology $(n-1)$-tori.\qed
\end{thm}

Similarly in the case of connected fixed point sets, we have the following.
\begin{prop}
If $X^{n-1}$ is a $\bZ[\bZ^{n-1}]$-homology $(n-1)$-torus bounding a strong (irreducible) $\bZ[\bZ^{n-1}]$-homology M\"obius band $W^{n}$, then the group $G=C_{2}$  acts on the $n$-torus $T^{n}$ with fixed point set  homeomorphic to $X^{n-1}$, and with orbit space $W^{n}$.
\end{prop}
\begin{proof}
The orientable double covering of $W^{n}$ clearly admits a fixed-point-free involution interchanging two copies of $X^{n-1}$, and with orbit space $W^{n}$. Identifying the two copies of $X^{n-1}$ by the involution produces a closed manifold $V^{n}$ with involution having fixed point set $X^{n-1}$ and orbit space $W^{n}$. By construction $V^{n}$ has the homotopy type of an $n$-torus, hence is homeomorphic to the $n$-torus.
\end{proof}

\begin{thm}The set of equivariant homeomorphism classes of locally linear involutions on $T^{n}$ with connected, codimension-one
 fixed point sets is in one-to-one correspondence with the set of homeomorphism classes   $X^{n-1}$ of $\bZ[\bZ^{n-1}]$-homology $(n-1)$-tori.
\end{thm}
\begin{proof}
This follows from Proposition \ref{prop:type1homeomorphic}, since the orientable double covering (depending only on the corresponding M\"obius band with boundary $X^{n-1}$) then determines the action.
\end{proof}

\end{document}